\documentclass[12pt]{article}

\usepackage{latexsym,amsfonts,amsmath,amsthm,amssymb, epsfig}
\usepackage{color}

\newtheorem{thm}{Theorem}
\newtheorem{lem}[thm]{Lemma}

\newtheorem{remark}{Remark}

\theoremstyle{remark}

\theoremstyle{definition}

\newcommand{\R}{\mathbb{R}}

\newcommand{\N}{\mathbb{N}}

\newcommand{\rd}{\,{\rm d}}
\newcommand{\bsx}{{\boldsymbol x}}
\newcommand{\bst}{{\boldsymbol t}}
\newcommand{\bszero}{{\boldsymbol 0}}
\newcommand{\uu}{\mathfrak{u}}
\newcommand{\vv}{\mathfrak{v}}
\newcommand{\cP}{\mathcal{P}}
\newcommand{\cA}{\mathcal{A}}

\title{The curse of dimensionality for the $L_p$-discrepancy with finite $p$}  

\author{Erich Novak and  Friedrich Pillichshammer}

\date{\today} 

\begin{document}

\maketitle

\begin{abstract}
The $L_p$-discrepancy is a quantitative measure for the irregularity of distribution of an $N$-element point set in the $d$-dimensional unit cube, which is closely related to the worst-case error of quasi-Monte Carlo algorithms for numerical integration. Its inverse for dimension $d$ and error threshold $\varepsilon \in (0,1)$ is the minimal number of points in $[0,1)^d$ such that the minimal normalized $L_p$-discrepancy is less or equal $\varepsilon$. It is well known, that the inverse of $L_2$-discrepancy grows exponentially fast with the dimension $d$, 
i.e., we have the curse of dimensionality, whereas the inverse of $L_{\infty}$-discrepancy depends exactly linearly on $d$. The behavior of inverse of $L_p$-discrepancy for general $p \not\in \{2,\infty\}$ has been an open problem for many years. In this paper we show that the $L_p$-discrepancy suffers from the curse of dimensionality for all $p$ in $(1,2]$ which are of the form $p=2 \ell/(2 \ell -1)$ with $\ell \in \mathbb{N}$. 

This result follows from a more general result that we show for the worst-case error of numerical integration in an anchored  Sobolev space with anchor 0 of once differentiable functions in each variable whose first derivative has finite $L_q$-norm, where $q$ is an even positive integer 
satisfying $1/p+1/q=1$.  
\end{abstract}

\centerline{\begin{minipage}[hc]{130mm}{
}
\end{minipage}}

\section{Introduction and main result}\label{sec:intro}

For a set $\cP$ consisting of $N$ points $\bsx_1,\bsx_2,\ldots,\bsx_N$ in the $d$-dimensional unit-cube $[0,1)^d$ the local discrepancy function $\Delta_{\cP}:[0,1]^d \rightarrow \R$ is defined as $$\Delta_{\cP}(\bst)=\frac{|\{k \in \{1,2,\ldots,N\}\ : \ \bsx_k \in [\bszero,\bst)\}|}{N}-{\rm volume}([\bszero,\bst)),$$  
for $\bst=(t_1,t_2,\ldots,t_d)$ in $[0,1]^d$, where $[\bszero,\bst)=[0,t_1)\times [0,t_2)\times \ldots \times [0,t_d)$. For a parameter $p \in [1,\infty]$ the $L_p$-discrepancy of the point set $\cP$ is defined as the $L_p$-norm of the local discrepancy function $\Delta_{\cP}$, i.e., $$L_{p,N}(\cP):=\left(\int_{[0,1]^d} |\Delta_{\cP}(\bst)|^p \rd \bst\right)^{1/p} \quad \mbox{for $p \in [1,\infty)$,}$$ and $$L_{\infty,N}(\cP):=\sup_{\bst \in [0,1]^d} | \Delta_{\cP}(\bst)| \quad  \mbox{for $p=\infty$.}$$ Traditionally, the $L_{\infty}$-discrepancy is called star-discrepancy and is denoted by $D_N^{\ast}(\cP)$ rather than $L_{\infty,N}(\cP)$. The study of $L_p$-discrepancy has its roots in the theory of uniform distribution modulo one; see \cite{BC,DT97,kuinie,mat} for detailed information. It is of particular importance because of its close relation to numerical integration. We will refer to this issue in Section~\ref{sec:int}.

Since one is interested in point sets with $L_p$-discrepancy as low as possible it is obvious to study for $d,N \in \N$ the quantity $${\rm disc}_p(N,d):=\min_{\cP} L_p(\cP),$$ where the minimum is extended over all $N$-element point sets $\cP$ in $[0,1)^d$. This quantity is called the $N$-th minimal $L_p$-discrepancy in dimension $d$.

Traditionally, the $L_p$-discrepancy is studied from the point of view of a fixed dimension $d$ and one asks for the asymptotic behavior for increasing sample sizes $N$. 
The celebrated result of Roth~\cite{Roth} is the most famous result in this direction and can be seen as the initial point of discrepancy theory. Later, Schmidt~\cite{sch77} extended Roth's lower bound to arbitrary $p>1$. For $p \in (1,\infty)$ it is known that for every dimension $d \in \N$ there exist positive reals $c_{d,p},C_{d,p}$ such that for every $N \ge 2$ it holds true that $$c_{d,p} \frac{(\log N)^{\frac{d-1}{2}}}{N} \le {\rm disc}_p(N,d) \le C_{d,p} \frac{(\log N)^{\frac{d-1}{2}}}{N}.$$ The upper bound was proven by Davenport~\cite{D56} for $p=2$, $d= 2$, by Roth~\cite{R80} for $p= 2$ and arbitrary $d$ and finally by Chen \cite{C80} in the general case. For more details see \cite{BC}. Explicit constructions of point sets can be found in \cite{CS02,D56,D14,DP14a,M15,S06}.

Similar results, but less accurate, are available also for $p\in \{1,\infty\}$. See the above references for further information. The currently best 
asymptotical lower bound in the $L_\infty$-case can be found in \cite{BLV}.  

All the classical bounds have a poor dependence on the dimension $d$. 
For large $d$ these bounds are only meaningful in an asymptotic sense 
(for very large $N$) and do not give any information about the 
discrepancy in the pre-asymptotic regime (see, e.g., \cite{NW08,NW10} or \cite[Section~1.7]{DKP} for discussions). Nowadays, motivated by applications of point sets with low discrepancy for numerical integration, there is dire need of information about the dependence of discrepancy on the dimension. 

This problem is studied with the help of the so-called inverse of $L_p$-discrepancy (or, in a more general context, the information complexity; see Section~\ref{sec:int}). This concept compares the minimal $L_p$-discrepancy with the initial discrepancy $${\rm disc}_p(0,d):=\left(\int_{[0,1]^d} \left({\rm volume}([\bszero,\bst))\right)^p \rd \bst\right)^{1/p},$$ which can be interpreted as the $L_p$-discrepancy of the empty point set, and asks for the minimal number $N$ of nodes that is necessary in order to achieve that the $N$-th minimal $L_p$-discrepancy is smaller than $\varepsilon$ times ${\rm disc}_p(0,d)$ for a threshold $\varepsilon \in (0,1)$. In other words, for $d \in \N$ and $\varepsilon \in (0,1)$ the inverse of the minimal $L_p$-discrepancy is defined as $$N_p^{{\rm disc}}(\varepsilon,d):=\min\{N \in \N \ : \ {\rm disc}_p(N,d) \le \varepsilon \ {\rm disc}_p(0,d)\}.$$ The question is now how fast $N_p^{{\rm disc}}(\varepsilon,d)$ increases, when $d \rightarrow \infty$ and $\varepsilon \rightarrow 0$. 

It is well known and easy to check that for the initial $L_p$-discrepancy we have 
\begin{equation}\label{initdisc}
{\rm disc}_p(0,d)=\left\{ 
\begin{array}{ll}
\frac{1}{(p+1)^{d/p}} & \mbox{if $p \in [1,\infty)$},\\[0.5em]
1 & \mbox{if $p=\infty$}.
\end{array}
\right.
\end{equation}
Here we observe a difference in the cases of finite and infinite $p$. While for $p=\infty$ the initial discrepancy equals 1 for every dimension $d$, for finite values of $p$ the initial discrepancy tends to zero exponentially fast with the dimension. 
 
For $p \in \{2,\infty\}$ the behavior of the inverse of the minimal $L_p$-discrepancy is well understood. In the $L_2$-case it is known that for all $\varepsilon \in (0,1)$ we have $$(1.125)^d (1-\varepsilon^2) \le N_2^{{\rm disc}}(\varepsilon,d) \le 1.5^d \varepsilon^{-2}.$$ Here the lower bound was first shown by Wo\'{z}niakowski in \cite{Wo99} (see also \cite{NW01,SW1998}) and the upper bound follows from an easy averaging argument, see, e.g., \cite[Sec.~9.3.2]{NW10}. 

In the $L_{\infty}$-case it was shown by Heinrich, Novak, Wasilkowski and Wo\'{z}nia\-kow\-ski in \cite{hnww} that there exists an absolute positive constant $C$ such that for every $d \in \N$ and $\varepsilon \in (0,1)$ we have $$N_{\infty}^{{\rm disc}}(\varepsilon,d) \le C d \varepsilon^{-2}.$$ The currently smallest known value of $C$ is $6.23401\ldots$ as shown in \cite{GPW} (thereby improving on other results from \cite{A11,D21,GH21,PW20}). On the other hand, there exist numbers $c>0$ and $\varepsilon_0 \in (0,1)$ such that for all $d \in \N$ and all $\varepsilon \in (0,\varepsilon_0)$ we have $$N_{\infty}^{{\rm disc}}(\varepsilon,d) \ge c d \varepsilon^{-1},$$ as shown by Hinrichs in \cite{Hi04}.

So while the inverse of $L_2$-discrepancy grows exponentially fast with the dimension $d$, the inverse of star-discrepancy depends only linearly on the dimension $d$. One says that the $L_2$-discrepancy suffers from the curse of dimensionality. In information based complexity theory the behavior of the inverse of star-discrepancy is called ``polynomial tractability'' (see, e.g., \cite{NW08,NW10}).

As we see, the situation is clear (and quite different) for $p \in \{2,\infty\}$. But what happens for all other $p \not \in \{2,\infty\}$? This question has been open for many years.

Just for completeness we remark that the problem has been considered also for other (semi) norms of the local discrepancy function rather than $L_p$-norms,  that are in some sense ``close'' to the $L_{\infty}$-norm. Positive results (tractability) have been obtained for exponential Orlicz norms in \cite{DHPP}. On the other hand, for the BMO-seminorm the curse of dimensionality has been shown in \cite{P22}.

All known proofs for lower bounds on the inverse of $L_2$-discrepancy are based on Hilbert-space methods. A very powerful tool in this context developed in \cite{NW01} (see also \cite[Chapter~12]{NW10}) is the method of decomposable reproducing kernels. Unfortunately, it is not obvious how these $L_2$-methods could be applied to the general $L_p$-case directly. So one has to find a new way or one has to figure out the essence of the Hilbert-space based proofs with the hope to get rid of all $L_2$-specific factors in order to find a way to extend these proofs to the general $L_p$-case. We shall follow the latter path. 

Our main result is formulated in the following theorem.

\begin{thm}\label{thm1}
For every $p$ of the form $p=\frac{2 \ell}{2 \ell-1}$ with $\ell \in \mathbb{N}$ there exists a real $C_p$ that is strictly larger than 1, such that for all $\varepsilon \in (0,1)$ we have $$N_p^{{\rm disc}}(\varepsilon,d) \ge C_p^d(1+o(1)) \quad \mbox{for $d \rightarrow \infty$}.$$ In particular, for all these $p$ the $L_p$-discrepancy suffers from the curse of dimensionality.
\end{thm}

At first glance, one would think that the result could be easily extended to any number $p \in (1,2]$ by means of monotonicity of the $L_p$-norm and squeezing any $p$ between two values of the form given in the theorem above. However, notice that we have to take the normalized discrepancy into account, which destroys monotonicity. Having a more careful look at the problem shows that it might be not so easy to follow this first intuition and in fact, so far we did not succeed in showing the curse of dimensionality for all $p \in (1,2]$. 

Theorem~\ref{thm1} will follow from a more general result about the integration problem in the anchored Sobolev space with a $q$-norm that will be introduced and discussed in the following Section~\ref{sec:int}. This result will be stated as Theorem~\ref{thm2}.

Sometimes a generalized notion of $L_p$-discrepancy is studied, where every point $\bsx_k$ is equipped with an own weight $a_k$  (rather than the weight $1/N$ for every point). The result from Theorem~\ref{thm1} even holds for this generalized $L_p$-discrepancy. This more general conclusion will be stated in Section~\ref{sec:rem} as Theorem~\ref{thm3}.

We add a conjecture: We guess that the curse holds for all $p$ with $1<p<\infty$.

\section{Relation to numerical integration}\label{sec:int}

It is well known that the $L_p$-discrepancy is related to multivariate integration (see, e.g., \cite[Chapter~9]{NW10}). From now on let $p,q \ge 1$ be such that $1/p+1/q=1$. For $d=1$ let $W_q^1([0,1])$ be the space of absolutely continuous functions whose first derivatives belong to the space $L_q([0,1])$. For $d>1$ consider the $d$-fold tensor product space which is denoted by $$W_q^{\boldsymbol{1}}:=W_q^{(1,1,\ldots,1)}([0,1]^d)$$ and which is the Sobolev space of functions on $[0,1]^d$ that are once differentiable in each variable and whose first derivative $\partial^d f/\partial \bsx$ has finite $L_q$-norm, where $\partial \bsx=\partial x_1 \partial x_2 \ldots \partial x_d$. Now consider the subspace of functions that satisfy the boundary conditions $f(\bsx)=0$ if at least one component of $\bsx=(x_1,\ldots,x_d)$ equals 0 and equip this subspace with the norm $$\|f\|_{d,q}:=\left(\int_{[0,1]^d} \left|\frac{\partial^d}{\partial \bsx}f(\bsx)\right|^q \rd \bsx \right)^{1/q} \quad \mbox{for $q \in [1,\infty)$,}$$ and $$\|f\|_{d,\infty}:=\sup_{\bsx \in [0,1]^d}\left|\frac{\partial^d}{\partial \bsx}f(\bsx)\right|  \quad \mbox{for $q=\infty$.}$$ That is, consider the space $$F_{d,q}:=\{f \in W_q^{\boldsymbol{1}} \ : \ f(\bsx)=0\ \mbox{if $x_j=0$ for some $j \in [d]$ and $\|f\|_{d,q}< \infty$}\},$$ where here and throughout this paper we write $[d]:=\{1,2,\ldots,d\}$.

Now consider multivariate integration $$I_d(f):=\int_{[0,1]^d} f(\bsx) \rd \bsx \quad \mbox{for $f \in F_{d,q}$}.$$
We approximate the integrals $I_d(f)$ by algorithms based on $N$ function evaluations of the form 
\begin{equation}\label{def:Alg}
A_{d,N}(f)=\varphi(f(\bsx_1),f(\bsx_2),\ldots,f(\bsx_N)),
\end{equation}
where $\varphi:\R^N \rightarrow \R$ is an arbitrary function and where $\bsx_1,\bsx_2,\ldots,\bsx_N$ are points in $[0,1]^d$, called integration nodes. Since $f\equiv 0$ belongs to $F_{d,q}$ we choose $\varphi$ such that $\varphi(0,0,\ldots,0)=0$. Typical examples are linear algorithms of the form 
\begin{equation}\label{def:linAlg}
A_{d,N}^{{\rm lin}}(f)=\sum_{k=1}^N a_j f(\bsx_k),
\end{equation}
where $\bsx_1,\bsx_2,\ldots,\bsx_N$ are in $[0,1]^d$ and $a_1,a_2,\ldots,a_N$ are real weights that we call integration weights. If $a_1=a_2=\ldots =a_N=1/N$, then the linear algorithm \eqref{def:linAlg} is a so-called quasi-Monte Carlo algorithm, for which we will write $A_{d,N}^{{\rm QMC}}$.

Define the worst-case error of an algorithm \eqref{def:Alg} by 
\begin{equation}\label{eq:wce}
e(F_{d,q},A_{d,N})=\sup_{f \in F_{d,q}\atop \|f\|_{d,q}\le 1} \left|I_d(f)-A_{d,N}(f)\right|.
\end{equation}
For a quasi-Monte Carlo algorithm $A_{d,N}^{{\rm QMC}}$ it is well known (see, e.g., \cite{SW1998}) that 
$$
e(F_{d,q},A_{d,N}^{{\rm QMC}})= L_p(\overline{\cP}),
$$ 
where $L_p(\overline{\cP})$ is the $L_p$-discrepancy of the point set\footnote{We remark that in \cite{SW1998} the anchored space with anchor $\boldsymbol{1}$ is considered which results in a worst case error of exactly
$L_p(\cP)$, where $\cP$ is the node set of the QMC rule. Here we have chosen the anchor as $\boldsymbol{0}$, and therefore
in the formula for the worst-case error the point set $\overline{\cP}$ appears. For details see also \cite[Section~9.5.1]{NW10} for the case $p=2$.} 
\begin{equation}\label{def:oP}
\overline{\cP}=\{\boldsymbol{1} - \bsx_k \ : \ k=1,2,\ldots,N\},
\end{equation}
where $\boldsymbol{1} - \bsx_k$ is defined as the component-wise difference of the vector containing only ones and $\bsx_k$. From this point of view we now study the more general problem of numerical integration in $F_{d,q}$ rather than only the $L_p$-discrepancy (which corresponds to quasi-Monte Carlo algorithms -- although with suitably ``reflected'' points). 

We define the $N$-th minimal worst-case error as $$e_q(N,d):=\min_{A_{d,N}} |e(F_{d,q},A_{d,N})|$$ where the minimum is extended over all algorithms of the form \eqref{def:Alg} based on $N$ function evaluations along points $\bsx_1,\bsx_2,\ldots,\bsx_N$ from $[0,1)^d$. Note that for all $d,N \in \N$ we have 
\begin{equation}\label{ine:errdisc}
e_q(N,d) \le {\rm disc}_p(N,d).
\end{equation} 

The initial error is $$e_q(0,d)=\sup_{f \in F_{d,q}\atop \|f\|_{d,q}\le 1} \left|I_d(f)\right|.$$ 

\begin{lem}\label{le:interr}
Let $d \in \N$ and let $q \in (1,\infty]$ and $p \in [1,\infty)$ with $1/p+1/q=1$. Then we have $$e_q(0,d)=\frac{1}{(p+1)^{d/p}}$$ and the worst-case function in $F_{d,q}$ is given by $h_d(\bsx)=h_1(x_1) \cdots h_1(x_d)$ for $\bsx=(x_1,\ldots,x_d) \in [0,1]^d$, where $h_1(x)=1-(1-x)^p$.
\end{lem}

\begin{proof}
Since we are dealing with tensor products of one-dimensional spaces it suffices to prove the result for $d=1$. For $f \in F_{1,q}$ we have $$\int_0^1 f(x) \rd x = \int_0^1 \int_0^x f'(t) \rd t \rd x = \int_0^1 f'(t) g(t) \rd t,$$ where $g(t)=1-t$.  Applying H\"older's inequality we obtain 
$$\left|\int_0^1 f(x) \rd x \right| \le \|f'\|_{L_q} \|g\|_{L_p}=\|f\|_{1,q} \|g\|_{L_p}$$ with equality if $f'(t)=c (g(t))^{p-1}=c (1-t)^{p-1}$ for some $c$. This holds for $f(t) = h_1(t):=1-(1-t)^p$. 

We have $$\int_0^1 h_1(t)\rd t= \frac{p}{p+1}$$ and $$\|h_1\|_{1,q}=  
\left(\int_0^1 (px^{p-1})^q \rd x \right)^{1/q} = p  \left( \frac{1}{(p-1)q+1} \right)^{1/q} = p (1+p)^{1/p-1}.$$ 
Hence, $$e_q(0,1)=\int_0^1 \frac{h_1(t)}{\|h_1\|_{1,q}} \rd t = \frac{1}{(p+1)^{1/p}}.$$
\end{proof}

Note that for all  $q \in (1,\infty]$ and $p \in [1,\infty)$ with $1/p+1/q=1$ and for all $d \in \N$ we have $$e_q(0,d)={\rm disc}_p(0,d).$$

Now we define the information complexity as the minimal number of function evaluations necessary in order to reduce the initial error by a factor of $\varepsilon$. For $d \in \N$ and $\varepsilon \in (0,1)$ put $$N^{{\rm int}}_q(\varepsilon,d):= \min\{N \in \N \ : \ e_q(N,d) \le \varepsilon\ e_q(0,d)\}.$$

From \eqref{ine:errdisc}, \eqref{initdisc} and Lemma~\ref{le:interr} it follows that for all  $q \in (1,\infty]$ and $p \in [1,\infty)$ with $1/p+1/q=1$ and for all $d \in \N$ and $\varepsilon \in (0,1)^d$ we have  $$N^{{\rm int}}_q(\varepsilon,d) \le N^{{\rm disc}}_p(\varepsilon,d).$$

Hence, Theorem~\ref{thm1} follows from the following more general result.

\begin{thm}\label{thm2}
For every positive even integer $q$ there exists a real $C=C(q)$, which depends on $q$ and which is strictly larger than 1, such that for all $\varepsilon \in (0,1)$ we have
\begin{equation}\label{lbd:C}
N_q^{{\rm int}}(\varepsilon,d) \ge C^d(1+o(1)) \quad \mbox{for $d \rightarrow \infty$}.
\end{equation}
In particular, for all even integer $q$ the integration problem in $F_{d,q}$ suffers from the curse of dimensionality.
\end{thm}

The proof of this result will be given in the following section.

Information about the base $C=C(q)$ of the exponentiation in Theorem~\ref{thm2} will be given in Section~\ref{sec:rem}.

\section{The proof of Theorem~\ref{thm2}}\label{sec:proof}

The proof of Theorem~\ref{thm2} is based on a suitable decomposition of the worst-case function $h_1$ from Lemma~\ref{le:interr}. This decomposition depends on $q$ and $p$, respectively, and will determine the base $C$ of the exponentiation in \eqref{lbd:C}. 

\begin{proof}[Proof of Theorem~\ref{thm2}] Let $q$ be an even, positive interger and let $p$ be such that $1/p+1/q=1$. Obviously, then $p=q/(q-1)$ and $p \in (1,2]$. According to Lemma~\ref{le:interr} the univariate worst-case function equals $$h_1(x)=1-(1-x)^p$$ and $$\|h_1\|_{1,q}=\frac{p}{(p+1)^{(p-1)/p}}.$$ Now we decompose $h_1$ in a suitable way into a sum of three functions. For real parameters $c$ and $a$, with $a \in (0,1)$ and $c \in (0,a/11)$ that will be determined later put
$$h_{1,2,(0)}(x) = \left\{ 
\begin{array}{ll}
\frac{1-(1-x)^{p}}{2} & \mbox{if $x \in [0,c)$,}\\[0.5em]
\frac{1-(1-c)^{p}}{2} & \mbox{if $x \in [c,a-10 c)$,}\\[0.5em]
\frac{1-(1-c)^{p}}{20 c}(a-x) & \mbox{if $x \in [a-10 c,a)$,}\\[0.5em]
0 & \mbox{if $x \in [a,1]$,}
\end{array}
\right. $$
and 
$$h_{1,2,(1)}(x) = {\bf 1}_{[a,1]}(x) \left((1-a)^{p}-(1-x)^{p}\right),$$
and 
$$h_{1,1}(x)=\left\{ 
\begin{array}{ll}
1-(1-x)^{p} - h_{1,2,(0)}(x)  & \mbox{if $x \in [0,a)$,}\\[0.5em]
1-(1-a)^{p}  & \mbox{if $x \in [a,1]$.}
\end{array}
\right.
$$

Then we have  $$h_{1,1}(x)+h_{1,2,(0)}(x)+h_{1,2,(1)}(x) =1-(1-x)^{p} =h_1(x)$$ and $h_{1,2,(0)}(x)$ and $h_{1,2,(1)}(x)$ have disjoint support as well as the derivatives $h_{1,1}'$ and $h_{1,2,(1)}'$. See Figure~\ref{fig_ca} for a graphical illustration of the decomposition.

\begin{figure}
\begin{center}
\includegraphics[width=10cm]{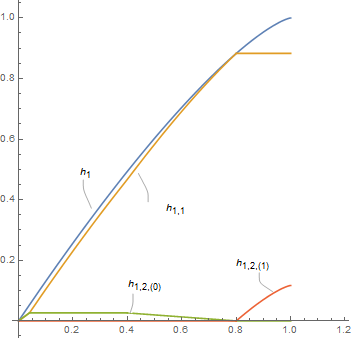}
\caption{Plot of the functions $h_1, h_{1,1}, h_{1,2,(0)}$, and $h_{1,2,(1)}$ for $q=4$ (hence $p=4/3$) and $a=0.8$ and $c=0.04$ (later we will use a different choice of parameters; but with the ones chosen here one gets a better impression of the principle situation).}
\label{fig_ca}
\end{center}
\end{figure}

Let $d \in \mathbb{N}$. For $\uu \subseteq [d]$ let $\uu^c= [d]\setminus \uu$. Now, for $\bsx=(x_1,\ldots,x_d) \in [0,1]^d$ let 
\begin{eqnarray*}
h_d(\bsx) & = & \prod_{j=1}^d h_1(x_j)\\
& = & \prod_{j=1}^d(h_{1,1}(x_j)+h_{1,2,(0)}(x_j)+h_{1,2,(1)}(x_j))\\
& = & \sum_{\uu \subseteq [d]} \prod_{j \in \uu^c} h_{1,1}(x_j) \sum_{\vv \subseteq \uu} \prod_{j \in \vv} h_{1,2,(0)}(x_j) \prod_{j \in \uu \setminus \vv} h_{1,2,(1)}(x_j).
\end{eqnarray*}
We have $$\|h_d\|_{d,q}=\|h_1\|_{1,q}^d=\left(\frac{p}{(p+1)^{(p-1)/p}}\right)^d.$$ 

For a point $\bsx \in [0,1]^d$ and for $\uu \subseteq [d]$ let $\bsx(\uu)$ denote the projection of $\bsx$ on the coordinates $j \in \uu$.

Now let $N \in \mathbb{N}$ and let $\mathcal{P}=\{\bsx_1,\bsx_2,\ldots,\bsx_N\}$ be a set of $N$ quadrature nodes in $[0,1]^d$. 
For this specific choice define the ``fooling function''
\begin{equation}\label{def:fool}
g_d(\bsx)=\sum_{\uu \subseteq [d]} \prod_{j \not\in \uu} h_{1,1}(x_j) \sum_{\vv \subseteq \uu} \hspace{-.05cm}\Big.^* \prod_{j \in \vv} h_{1,2,(0)}(x_j) \prod_{j \in \uu \setminus \vv} h_{1,2,(1)}(x_j),
\end{equation}
where the $\ast$ in $ \sum^*_{\vv \subseteq \uu}$ indicates that the summation is restricted to all $\vv \subseteq \uu$ such that for all $k \in \{1,2,\ldots,N\}$ it holds that 
$$
(\bsx_k(\vv),\bsx_k(\uu \setminus \vv)) \not \in  [0,a]^{|\vv|}\times [a,1]^{|\uu|-|\vv|}.
$$ 
This condition means that every point from the set $\mathcal{P}$ is outside the support of $\prod_{j \not\in \uu} h_{1,1}(x_j)\prod_{j \in \vv} h_{1,2,(0)}(x_j) \prod_{j \in \uu \setminus \vv} h_{1,2,(1)}(x_j)$ and this guarantees that 
$$g_d(\bsx_k)=0 \quad \mbox{for all $k \in \{1,2,\ldots,N\}$.}$$ 
Observe that the fooling function $g_d$ is an element of the $3^d$-dimensional 
(tensor product) subspace generated, for $d=1$, by the three functions 
$h_{1,1}$, $h_{1,2,(0)}$ and $h_{1,2,(1)}$. Formally, we prove the lower bound for this 
finite-dimensional subspace. 
We use here notation from the theory of decomposable kernels (see, e.g., \cite[Section~11.5]{NW10}). 
This theory can only be applied for $q=p=2$ 
and then one can choose the three functions orthogonal which 
simplifies the proof a lot (see also Section~\ref{sec:q2}).

Now we show that $\|g_d\|_{d,q} \le \|h_d\|_{d,q}$. This is the only part of the proof where we need that $q$ is an even integer.  We have
\begin{eqnarray*}
\lefteqn{\|h_d\|_{d,q}^q }\\
& = & \int_{[0,1]^d}\left|\sum_{\uu \subseteq [d]} \prod_{j \not\in \uu} h_{1,1}'(x_j) \sum_{\vv \subseteq \uu} \prod_{j \in \vv} h_{1,2,(0)}'(x_j) \prod_{j \in \uu \setminus \vv} h_{1,2,(1)}'(x_j)\right|^q \rd \bsx\\
& = & \sum_{\uu_1 \subseteq [d]}\sum_{\vv_1 \subseteq \uu_1} \sum_{\uu_2 \subseteq [d]}\sum_{\vv_2 \subseteq \uu_2}   \ldots  \sum_{\uu_q \subseteq [d]}\sum_{\vv_q \subseteq \uu_q} \int_{[0,1]^d}\\
&& \hspace{1cm}\prod_{j \in \uu_1^c}h_{1,1}'(x_j) \prod_{j \in \uu_2^c}h_{1,1}'(x_j) \ldots  \prod_{j \in \uu_q^c}h_{1,1}'(x_j) \\
&& \hspace{1cm}\prod_{j \in \vv_1}h_{1,2,(0)}'(x_j) \prod_{j \in \vv_2}h_{1,2,(0)}'(x_j) \ldots \prod_{j \in \vv_q}h_{1,2,(0)}'(x_j) \\
&& \hspace{1cm}\prod_{j \in \uu_1\setminus\vv_1}h_{1,2,(1)}'(x_j) \prod_{j \in \uu_2\setminus\vv_2}h_{1,2,(0)}'(x_j) \ldots  \prod_{j \in \uu_q\setminus\vv_q}h_{1,2,(1)}'(x_j) \rd \bsx\\
& = & \sum_{\uu_1 \subseteq [d]}\sum_{\vv_1 \subseteq \uu_1} \sum_{\uu_2 \subseteq [d]}\sum_{\vv_2 \subseteq \uu_2}   \ldots  \sum_{\uu_q \subseteq [d]}\sum_{\vv_q \subseteq \uu_q} \prod_{j=1}^d  \int_0^1 h_{1,1}'(x)^{\alpha_j} h_{1,2,(0)}'(x)^{\beta_j}   h_{1,2,(1)}'(x)^{\gamma_j} \rd x,
\end{eqnarray*}
where, for $j \in [d]$, we put 
\begin{eqnarray*}
\alpha_j & = &  |\{t \in \{1,2,\ldots,q\}\ : \ j \in \uu_t^c\}|,\\
\beta_j& = & |\{t \in \{1,2,\ldots,q\}\ : \ j \in \vv_t\}|,\\
\gamma_j& = & |\{t \in \{1,2,\ldots,q\}\ : \ j \in \uu \setminus \vv_t\}|.
\end{eqnarray*}
Note that $\alpha_j,\beta_j,\gamma_j$ also depend on $\uu_1,\vv_1,\uu_2,\vv_2,\ldots,\uu_q,\vv_q$. We have $\alpha_j,\beta_j,\gamma_j \in \{0,1,\ldots,q\}$ and $\alpha_j+\beta_j+\gamma_j=q$. 

For given $\uu_t \subseteq [d], \vv_t \subseteq \uu_t$ for $t \in \{1,2,\ldots,q\}$, put $$I(\alpha_j,\beta_j, \gamma_j)= \int_0^1 h_{1,1}'(x)^{\alpha_j} h_{1,2,(0)}'(x)^{\beta_j}   h_{1,2,(1)}'(x)^{\gamma_j} \rd x.$$
Now we show that $I(\alpha_j,\beta_j,\gamma_j)\ge 0$ for all $\alpha_j,\beta_j,\gamma_j \in \{0,1,\ldots,q\}$ with $\alpha_j+\beta_j+\gamma_j=q$. This will help us later on when we show that the norm of $g_d$ is dominated by the norm of $h_d$.

If $\alpha_j>0$ and $\gamma_j>0$, then we have $I(\alpha_j,\beta_j, \gamma_j)=0$, because $h_{1,1}'$ and $h_{1,2,(1)}'$ have disjoint support. Likewise, if $\beta_j>0$ and $\gamma_j>0$, then we have $I(\alpha_j,\beta_j, \gamma_j)=0$, because $h_{1,2,(0)}$ and $h_{1,2,(1)}$ have disjoint support. 

If $(\alpha_j,\beta_j,\gamma_j) \in \{(0,0,q),(0,q,0),(q,0,0)\}$, then obviously $I(\alpha_j,\beta_j,\alpha_j)>0$.

It remains to consider the case $\alpha_j>0,\beta_j>0$ and $\gamma_j=0$. In this case $\alpha_j+\beta_j=q$. If $\alpha_j$ (and then also $\beta_j$) is even, then  $I(\alpha_j,\beta_j, \gamma_j)>0$.

So it remains to deal with the cases where $\alpha_j$ is odd. 

Let $a^*=1-10^{-q}$ and $c^*=10^{-q}$. Let $\alpha,\beta \in \{1,\ldots,q-1\}$ be odd and assume that $\alpha+\beta=q$. We show that 
\begin{equation}\label{prop:Iab}
I(\alpha,\beta,0)\ge 0\quad\mbox{ for all $a \in [a^*,1]$ and all $c \in [0,c^*]$.}
\end{equation}

Let $a \in [a^*,1]$ and $c \in [0,c^*]$. We have
\begin{eqnarray*}
\lefteqn{I(\alpha,\beta,0)}\\
& = & \int_0^c \left(\frac{p}{2}(1-x)^{p-1}\right)^{\alpha} \left(\frac{p}{2}(1-x)^{p-1}\right)^{\beta} \rd x\\
& & + \int_c^{a-10c} 0 \rd x\\
& & + \int_{a-10 c}^a \left(p(1-x)^{p-1}+\frac{1-(1-c)^{p}}{20 c}\right)^{\alpha} \left(-\frac{1-(1-c)^{p}}{20 c}\right)^{\beta}\rd x\\
& = & \int_0^c \left(\frac{p}{2}(1-x)^{p-1}\right)^q  \rd x\\
& & - \left(\frac{1-(1-c)^{p}}{20 c}\right)^{\beta} \int_{a-10 c}^a \left(p(1-x)^{p-1}+\frac{1-(1-c)^{p}}{20 c}\right)^{\alpha}\rd x.
\end{eqnarray*}
Note that $(p-1)q=p$ and hence
\begin{eqnarray*}
\lefteqn{I(\alpha,\beta,0)}\\
& = & \left(\frac{p}{2}\right)^q \frac{1-(1-c)^{p+1}}{p+1} \\
& & - \left(\frac{1-(1-c)^{p}}{20 c}\right)^{\beta} \int_{a-10 c}^a \left(p(1-x)^{p-1}+\frac{1-(1-c)^{p}}{20 c}\right)^{\alpha}\rd x.
\end{eqnarray*}
We have $1-(1-c)^p \le p c$ and $1-a+10c \le 11 c^*$. Hence,
\begin{eqnarray*}
\lefteqn{\left(\frac{1-(1-c)^{p}}{20 c}\right)^{\beta} \int_{a-10 c}^a \left(p(1-x)^{p-1}+\frac{1-(1-c)^{p}}{20 c}\right)^{\alpha}\rd x}\\
& \le & \left(\frac{p}{20}\right)^{\beta} \int_{a-10 c}^a \left(p(1-x)^{p-1}+\frac{p}{20}\right)^{\alpha}\rd x\\
& \le & \left(\frac{p}{20}\right)^{\beta} 10c  \left(p(1-a+10c)^{p-1}+\frac{p}{20}\right)^{\alpha}\\
& = & \left(\frac{p}{2}\right)^q \frac{10c}{10^{\beta}}  \left(2(11c^*)^{p-1}+\frac{1}{10}\right)^{\alpha}\\
& \le & \left(\frac{p}{2}\right)^q \frac{10c}{10^q}  \left(\frac{32}{10}\right)^{q-1},
\end{eqnarray*}
where we used  
\begin{eqnarray*}
2(11c^*)^{p-1}+\frac{1}{10}\le 2 \frac{11^{p-1}}{10^{(p-1)q}}+\frac{1}{10} = 2 \frac{11^{p-1}}{10^p}+\frac{1}{10}=\frac{2}{11} \left(\frac{11}{10}\right)^p+\frac{1}{10}\le \frac{1}{10} \frac{32}{10}.
\end{eqnarray*}
Now we obtain
\begin{eqnarray*}
I(\alpha,\beta,0) \ge \left(\frac{p}{2}\right)^q \left(\frac{1-(1-c)^{p+1}}{p+1} - \frac{c}{10^{q-1}}  \left(\frac{32}{10}\right)^{q-1}\right).
\end{eqnarray*}
We use $1-(1-c)^{p+1} \ge (p+1) (1-c)^p c$. Then we get
\begin{eqnarray*}
I(\alpha,\beta,0) & \ge & \left(\frac{p}{2}\right)^q c \left((1-c)^p  - \frac{1}{10^{q-1}}  \left(\frac{32}{10}\right)^{q-1}\right)\\
& \ge & \left(\frac{p}{2}\right)^q c \left(\left(1-\frac{1}{10^q}\right)^p  -  \left(\frac{32}{100}\right)^{q-1}\right)\\
& = & \left(\frac{p}{2}\right)^q c \left(\left(1-\frac{1}{10^q}\right)^{q/(q-1)}  -  \left(\frac{32}{100}\right)^{q-1}\right)\\
& \ge &  \left(\frac{p}{2}\right)^q c \frac{6601}{10000} \ge 0.
\end{eqnarray*}
Hence \eqref{prop:Iab} is shown.

From now on let $a \in [a^*,1)$ and $c \in (0,c^*]$. We summarize: for $\alpha_j,\beta_j,\gamma_j \in \{0,1,\ldots,q\}$ and $\alpha_j+\beta_j+\gamma_j=q$ we have the following situation:
$$
\begin{array}{c|c|c||c|l}
\alpha_j & \beta_j & \gamma_j & I(\alpha_j,\beta_j,\gamma_j) & \mbox{reason}\\
\hline
=0 & =0 &  > 0 & >0 & \gamma_j=q \\
= 0 & > 0 & = 0 & >0 & \beta_j=q \\
0 & > 0 & > 0 & = 0 & \mbox{disjoint support of $h_{1,2,(0)}$ and $h_{1,2,(1)}$}\\
>0 & =0 & =0 & > 0 & \alpha_j=q\\
>0 & =0 & >0 & = 0 &  \mbox{disjoint support of $h_{1,1}'$ and $h_{1,2,(1)}'$}\\
>0 & > 0 & =0 & > 0 & \mbox{special choice $a \in [a^*,1)$ and $c \in (0,c^*]$}\\
>0 & >0 & >0 & =0 & \mbox{disjoint support of $h_{1,2,(0)}$ and $h_{1,2,(1)}$}
\end{array} 
$$
Thus, in any case we have $I(\alpha_j,\beta_j,\gamma_j)\ge 0$. Hence
\begin{eqnarray*}
\lefteqn{\|h_d\|_{d,q}^q }\\
& = & \sum_{\uu_1 \subseteq [d]}\sum_{\vv_1 \subseteq \uu_1}  \sum_{\uu_2 \subseteq [d]}\sum_{\vv_2 \subseteq \uu_2}  \ldots  \sum_{\uu_q \subseteq [d]}\sum_{\vv_q \subseteq \uu_q} \prod_{j=1}^d  I(\alpha_j,\beta_j,\gamma_j)\\
& \ge & \sum_{\uu_1 \subseteq [d]}\sum_{\vv_1 \subseteq \uu_1}\hspace{-.05cm}\Big.^*  \sum_{\uu_2 \subseteq [d]}\sum_{\vv_2 \subseteq \uu_2}\hspace{-.05cm}\Big.^*  \ldots  \sum_{\uu_q \subseteq [d]}\sum_{\vv_q \subseteq \uu_4}\hspace{-.05cm}\Big.^* \prod_{j=1}^d  I(\alpha_j,\beta_j,\gamma_j)\\
& = &  \int_{[0,1]^d}\left|\sum_{\uu \subseteq [d]} \prod_{j \not\in \uu} h_{1,1}'(x_j) \sum_{\vv \subseteq \uu}\hspace{-.05cm}\Big.^* \prod_{j \in \vv} h_{1,2,(0)}'(x_j) \prod_{j \in \uu \setminus \vv} h_{1,2,(1)}'(x_j)\right|^q \rd \bsx\\
& = & \|g_d\|_{d,q}^q.
\end{eqnarray*}

Now set $\widetilde{g}_d(\bsx):=(1/\|h_d\|_{d,q}) g_d(\bsx)$ such that $\|\widetilde{g}_d\|_{d,q} \le 1$. 

Put, for $(a,c) \in [a^*,1] \times [0,c^*]$, 
\begin{eqnarray*}
I_{1,2,(0)}(a,c) & := & \int_0^1 h_{1,2,(0)}(x)\rd x\\
& = & \frac{(1-(1-c)^p)((p+1)(a-5c)-1)+p c (1-c)^p}{2 (p+1)}
\end{eqnarray*}
and $$I_{1,2,(1)}(a) :=\int_0^1 h_{1,2,(1)}(x)\rd x=(1-a)^{p+1} \frac{p}{p+1}.$$

Next we show that we can choose $c \in (0,c^*)$ such that $I_{1,2,(0)}(a^*,c) =I_{1,2,(1)}(a^*)$. 

We have 
\begin{eqnarray*}
\frac{I_{1,2,(0)}(a^*,c^*)}{I_{1,2,(1)}(a^*)} & = & \frac{(1-(1-c^*)^p)((p+1)(a^*-5c^*)-1)+p c^* (1-c^*)^p}{2 (1-a^*)^{p+1} p}\\
& = & \frac{(1-(1-c^*)^p)((p+1)(1-6c^*)-1)+p c^* (1-c^*)^p}{2 (c^*)^{p+1} p}\\
& \ge & \frac{p (1-c^*)^{p-1}((p+1)(1-6c^*)-1)+p c^* (1-c^*)^p}{2 (c^*)^{p+1} p}\\
& = & \frac{(1-c^*)^{p-1}}{2 (c^*)^{p+1}}((p+1)(1-6c^*)-1+c^*(1-c^*))\\
& \ge & \frac{1-11c^*-(c^*)^2}{2 (c^*)^2}\\
&= & \tfrac{1}{2} 10^{2q}-\tfrac{11}{2} 10^q-\tfrac{1}{2} >1.
\end{eqnarray*}
Hence $I_{1,2,(0)}(a^*,c^*) > I_{1,2,(1)}(a^*)$. Now, since $I_{1,2,(0)}(a^*,0) =0 < I_{1,2,(1)}(a^*)$, we can find a $c \in (0,c^*)$ such that $I_{1,2,(0)}(a^*,c)=I_{1,2,(1)}(a^*)$, according to the mean value theorem.

From now on let 
\begin{equation}\label{fix:ac}
c \in (0,c^*) \mbox{ be such that $I_{1,2,(0)}(a^*,c) =I_{1,2,(1)}(a^*)$.}
\end{equation}

In the following we use $(x)_+=x$ if $x >0$ and 0 if $x \le 0$. Since $\widetilde{g}_d(\bsx_k)=0$ for all $k \in \{1,2,\ldots,N\}$ and since $\varphi(0,0,\ldots,0)=0$ (here $\varphi$ is from \eqref{def:Alg}) we have for any algorithm \eqref{def:Alg} that is based on $\cP$ that
\begin{eqnarray}\label{err_est1}
\lefteqn{e(F_{d,q},A_{d,N})}\nonumber\\
& \ge & \int_{[0,1]^d} \widetilde{g}_d(\bsx) \rd \bsx\nonumber\\
& = & \frac{1}{\|h_d\|_{d,q}} \sum_{\uu \subseteq [d]} \left(\int_0^1 h_{1,1}(x)\rd x\right)^{d-|\uu|} \nonumber \\
& & \hspace{1.8cm}\times \sum_{\vv \subseteq \uu} \hspace{-.05cm}\Big.^* \left(\int_0^1 h_{1,2,(0)}(x)\rd x\right)^{|\vv|} \left(\int_0^1 h_{1,2,(1)}(x) \rd x \right)^{|\uu|-|\vv|}.
\end{eqnarray}
 
Put
\begin{eqnarray*}
\alpha_1  := \int_0^1 h_{1,1}(x)\rd x= \frac{p}{p+1}(1-(1-a^*)^{p+1})-I_{1,2,(0)}(a^*,c)
\end{eqnarray*}
and remember that according to our choice \eqref{fix:ac} we have 
\begin{eqnarray*}
\int_0^1 h_{1,2,(0)}(x)\rd x = \int_0^1 h_{1,2,(1)}(x)\rd x.
\end{eqnarray*}

Now we continue from \eqref{err_est1} and obtain
\begin{eqnarray*}
\lefteqn{e(F_{d,q},A_{d,N})}\\
& \ge & \frac{1}{\|h_d\|_{d,q}} \sum_{\uu \subseteq [d]} \alpha_1^{d-|\uu|}  \left(\int_0^1 h_{1,2,(0)}(x)\rd x\right)^{|\uu|} \sum_{\vv \subseteq \uu} \hspace{-.05cm}\Big.^* 1\nonumber\\
& \ge & \frac{1}{\|h_d\|_{d,q}}  \sum_{\uu \subseteq [d]} \alpha_1^{d-|\uu|}  \left(\int_0^1 h_{1,2,(0)}(x)\rd x\right)^{|\uu|} (2^{|\uu|}-N)_+\\
&= & \frac{1}{\|h_d\|_{d,q}}  \sum_{\uu \subseteq [d]} \alpha_1^{d-|\uu|}  \left(2 \int_0^1 h_{1,2,(0)}(x)\rd x\right)^{|\uu|} \left(1-\frac{N}{2^{|\uu|}}\right)_+,
\end{eqnarray*}
where we used that for every $\uu \subseteq [d]$ at most $N$ of the intervals $$\prod_{j\in \vv}[0,a^*] \prod_{j \in \uu \setminus \vv}[a^*,1] \prod_{j \not \in \uu}[0,1]$$ with $\vv \subseteq \uu$ can contain a point from the node set $\mathcal{P}_N$.

Put $$\alpha_2:= 2 \int_0^1 h_{1,2,(0)}(x)\rd x.$$
Then we have further
\begin{eqnarray*}
e(F_{d,q},A_{d,N}) & \ge & \frac{1}{\|h_d\|_{d,q}}  \sum_{\uu \subseteq [d]} \alpha_1^{d-|\uu|}  \alpha_2^{|\uu|} \left(1-\frac{N}{2^{|\uu|}}\right)_+\\
& = & \frac{1}{\|h_d\|_{d,q}}  \sum_{k=0}^d {d \choose k} \alpha_1^{d-k}  \alpha_2^{k} \left(1-\frac{N}{2^{k}}\right)_+\\
& = &  \frac{1}{\|h_d\|_{d,q}}  \alpha_1^d \sum_{k=0}^d {d \choose k} \alpha_3^{k} \left(1-\frac{N}{2^{k}}\right)_+,
\end{eqnarray*}
where we put $$\alpha_3:=\frac{\alpha_2}{\alpha_1}>0.$$ 

Since the lower bound on the error is independent of the choice of $\mathcal{P}_N$ it follows that 
\begin{equation*}
e_q(N,d) \ge \frac{1}{\|h_d\|_{d,q}}\alpha_1^d \sum_{k=0}^d {d \choose k} \alpha_3^{k} \left(1-\frac{N}{2^{k}}\right)_+.
\end{equation*}
For the normalized $n$-th minimal error we therefore obtain
\begin{equation*}
\frac{e_q(N,d)}{e_q(0,d)} \ge \frac{1}{\|h_d\|_{d,q} \, e_q(0,d)}  \alpha_1^d \sum_{k=0}^d {d \choose k} \alpha_3^{k} \left(1-\frac{N}{2^{k}}\right)_+.
\end{equation*}
We have $$\|h_d\|_{d,q} \, e_q(0,d)=\left(\frac{p}{p+1}\right)^d$$ and hence
\begin{equation}\label{lb:ne}
\frac{e_q(N,d)}{e_q(0,d)} \ge \left(\frac{p+1}{p}\right)^d  \alpha_1^d \sum_{k=0}^d {d \choose k} \alpha_3^{k} \left(1-\frac{N}{2^{k}}\right)_+.
\end{equation}

Now we proceed using ideas from \cite[p.~185]{NW10}. Let $b \in (0,1)$ and let $N =\lfloor C^d\rfloor$ with $C \in (1,2^{\alpha_3/(1+\alpha_3)})$. This means that $$\frac{C}{2^{\alpha_3/(1+\alpha_3)}} < 1.$$ Then there exists a positive $c$ such that $$c < \frac{\alpha_3}{1+\alpha_3} \quad \mbox{ and }\quad \frac{C}{2^c}<1.$$ Put $k(d):= \lfloor c d\rfloor$. For sufficiently large $d$ we have 
\begin{equation}\label{eq:b}
\frac{N}{2^k} \le \frac{C^d}{2^{cd-1}}=2 \left(\frac{C}{2^c}\right)^d \le b \quad \mbox{ for all $k \in (k(d),d]$.}
\end{equation}
Put further $C_{d,k}:={d \choose k} \alpha_3^k$. Then \eqref{lb:ne} and \eqref{eq:b} imply that
\begin{eqnarray*}
\frac{e_q(\lfloor C^d\rfloor,d)}{e_q(0,d)} & \ge & \left(\frac{p+1}{p}\right)^d \alpha_1^d \sum_{k=k(d)+1 }^d C_{d,k} \left(1-\frac{N}{2^{k}}\right)_+\\
& \ge & (1-b)\left(\frac{p+1}{p}\right)^d \alpha_1^d \sum_{k=k(d)+1 }^d C_{d,k} \\
& = & (1-b) \left(\frac{p+1}{p}\right)^d \alpha_1^d \left(\sum_{k=0}^d C_{d,k} - \sum_{k=0}^{k(d)} C_{d,k} \right).
\end{eqnarray*}
Note that $$\sum_{k=0}^d C_{d,k} = (1+\alpha_3)^d =\frac{(\alpha_1+\alpha_2)^d}{\alpha_1^d}$$ and 
\begin{eqnarray*}
\alpha_1+\alpha_2 & = & \int_0^1 h_{1,1}(x) \rd x + 2 \int_0^1 h_{1,2,(0)}(x) \rd x\\
& = & \int_0^1 h_{1,1}(x) \rd x + \int_0^1 h_{1,2,(0)}(x) \rd x + \int_0^1 h_{1,2,(1)}(x) \rd x\\
& = & \int_0^1 h_1(x) \rd x\\
& = & \int_0^1 1-(1-x)^{p} \rd x\\
& = & \frac{p}{p+1}.
\end{eqnarray*}
This yields
\begin{eqnarray*}
\frac{e_q(\lfloor C^d\rfloor,d)}{e_q(0,d)} & \ge &  (1-b) \left(\frac{p+1}{p} (\alpha_1+\alpha_2)\right)^d  \left(1 -\frac{\sum_{k=0}^{k(d)} C_{d,k}}{(1+\alpha_3)^d} \right)\\
& = & (1-b)  \left(1 -\frac{\sum_{k=0}^{k(d)} C_{d,k}}{(1+\alpha_3)^d} \right).
\end{eqnarray*}

Now put $$\alpha(d):=\frac{\sum_{k=0}^{k(d)} C_{d,k}}{(1+\alpha_3)^d}.$$  It is shown in \cite[p.~185]{NW10} that $\alpha(d)$ tends to zero when $d \rightarrow \infty$.  Hence, for any positive $\delta$ we can find an integer $d(\delta)$ such that for all $d \ge d(\delta)$ we have $$1 \ge  \frac{e_q(\lfloor C^d\rfloor,d)}{e_q(0,d)} \ge (1-b)(1-\delta).$$ Since $b$ and $\delta$ can be arbitrarily close to zero, it follows that $$\lim_{d \rightarrow \infty} \frac{e_q(\lfloor C^d\rfloor,d)}{e_q(0,d)} =1, $$ and this holds true for all $C \in (1,2^{\alpha_3/(1+\alpha_3)})$.

Now we proceed like in \cite[p.~186]{NW10}. Take $\varepsilon \in (0,1)$. For large $d$ we have $$\varepsilon e_q(0,d) < e_q(\lfloor C^d\rfloor,d) = e_q(0,d) (1+o(1))$$ and therefore $N_q^{{\rm int}}(\varepsilon,d) \ge C^d (1+o(1))$. Since $C>1$ this means the curse of dimensionality. 
\end{proof}

\section{Some remarks}\label{sec:rem}

Some remarks are in order.

\paragraph{Possible values for $C$ in Theorem~\ref{thm2}.} According to the proof of Theorem~\ref{thm2} we have $$\alpha_3=\frac{\alpha_2}{\alpha_1}=\frac{2 I_{1,2,(0)}(a^*,c)}{\frac{p}{p+1}(1-(1-a^*)^{p+1})-I_{1,2,(0)}(a^*,c)},$$ where $c$ is choosen according to \eqref{fix:ac}, which implies $$I_{1,2,(0)}(a^*,c)=I_{1,2,(1)}(a^*)=\frac{p}{p+1} (1-a^*)^{p+1}.$$ Hence, $$\alpha_3=\frac{2(1-a^*)^{p+1}}{1-2(1-a^*)^{p+1}}$$ and further $$\frac{\alpha_3}{1+\alpha_3}=2(1-a^*)^{p+1} =\frac{2}{10^{q(p+1)}} \quad \mbox{ and }\quad 2^{\alpha_3/(1+\alpha_3)}=4^{10^{-q(p+1)}}.$$ 

We compute some values for $2^{\alpha_3/(1+\alpha_3)}$:
$$
\begin{array}{c|c}
q & 2^{\alpha_3/(1+\alpha_3)}\\
\hline
2 & 1.01396\ldots\\
4 & 1.000003482\ldots\\
6 & 1.00000000051675\ldots 
\end{array}
$$

\begin{remark}\rm
Our choice for $a^*$ was very arbitrary.  One can try to find the smallest possible $a^*$ such that all conditions are satisfied, i.e., $I(\alpha,1-\alpha,0)\ge$ for all odd $\alpha$ in $\{1,\ldots,q-1\}$ and $I_{1,2,(0)}=I_{1,2,(1)}$. This way one can improve the value of $2^{\alpha_3/(1+\alpha_3)}$. For example for $q=4$ we also found $a^*=0.930338256\ldots$ and $c=0.00186068\ldots$. With this choice we obtain $$2^{\alpha_3/(1+\alpha_3)}=1.00277\ldots.$$ Hence we can deduce that $$N_4^{{\rm int}}(\varepsilon,d) \ge (1.00277)^d (1+o(1)).$$
\end{remark}

\paragraph{History and related problems.} During the proof of Theorem~\ref{thm2} we already observed that, formally, we prove a lower bound for the integration problem and a $3^d$-dimen\-sional tensor product space. We give some history of such problems, in particular (in)tractability results. The first results appeared in 1997. In \cite{NSW97} it is proved that some non-trivial problems are easy (tractable) and for other problems the curse of dimensionality is present. An important case is the space of trigonometric polynomials of degree one in each variable;  Sloan and Wo\'zniakowski~\cite{SW97} prove the curse for this space. The same space, but with different norms, was studied later in \cite{NoX}. To prove the curse for this norm, a new technique was developed in \cite{Vy20} and further studied in \cite{HKNV21} and \cite{HKNV22}. 

As we explain in the next section, also the method of decomposable kernels (see, e.g., \cite[Theorem~11.12]{NW10})
can be seen as a contribution to this problem, even if the authors at that time did not present it that way. This technique works fine in the case of finite smoothness and is related to the technique of bump functions, see \cite{KV23} for a discussion and new results. Bump functions do not work for analytic functions, see again \cite{HKNV21,HKNV22}, and surprisingly they also do not work for functions of very low degree of smoothness, see~\cite{KV23}.

\paragraph{Generalized $L_p$-discrepancy.} 
Sometimes a more general version of $L_p$-discrepancy is considered. Here for points $\cP=\{\bsx_1,\bsx_2,\ldots,\bsx_N\}$ and corresponding coefficients $\cA=\{a_1,a_2,\ldots,a_N\}$ the discrepancy function is $$\overline{\Delta}_{\cP,\cA}(\bst)=\sum_{k=1}^N a_k {\bf 1}_{[0,\bst)}(\bsx_k) - t_1 t_2\cdots t_d$$ for $\bst=(t_1,t_2,\ldots,t_d)$ in $[0,1]^d$ and the generalized $L_p$-discrepancy is $$\overline{L}_{p,N}(\cP,\cA)=\left(\int_{[0,1]^d} |\overline{\Delta}_{\cP,\cA}(\bst)|^p \rd \bst\right)^{1/p} \quad \mbox{for $p \in [1,\infty)$,}$$ with the usual adaptions for $p=\infty$. If $a_1=a_2=\ldots=a_N=1/N$, then we are back to the classical definition of $L_p$-discrepancy from Section~\ref{sec:intro}. 

Now the $N$-th minimal generalized $L_p$-discrepancy in dimension $d$ is defined as $$\overline{{\rm disc}}_p(N,d):=\min_{\cP,\cA} \overline{L}_p(\cP,\cA)$$ where the minimum is extended over all $N$-element point sets $\cP$ in $[0,1)^d$ and over all corresponding coefficients $\cA$, and its inverse is $$N_p^{\overline{{\rm disc}}}(\varepsilon,d):=\min\{N \in \N \ : \ \overline{{\rm disc}}_p(N,d) \le \varepsilon \ \overline{{\rm disc}}_p(0,d)\},$$ where, obviously, $\overline{{\rm disc}}_p(0,d)={\rm disc}_p(0,d)$. 

Turn to the integration problem in $F_{d,q}$, then the worst-case error \eqref{eq:wce} of a linear algorithm \eqref{def:linAlg} is 
$$
e(F_{d,q},A_{d,N}^{{\rm lin}})= \overline{L}_p(\overline{\cP},\cA),
$$ 
where $\overline{\cP}$ is given in \eqref{def:oP} and $\cA$ consists of exactly the coefficients from the given linear algorithm. Again, this is well known (see \cite{NW10} or \cite{SW1998}).

With these preparations it follows from Theorem~\ref{thm2} that also the generalized $L_p$-discrepancy suffers for all 
$p$ of the form $p=\frac{2 \ell}{2 \ell -1}$ with $\ell \in \mathbb{N}$
from the curse of dimensionality.

\begin{thm}\label{thm3}
For every $p$ of the form $p=\frac{2 \ell}{2 \ell -1}$ with $\ell \in \mathbb{N}$ there exists a real $C_p$ that is strictly larger than 1, such that for all $\varepsilon \in (0,1)$ we have $$N_p^{\overline{{\rm disc}}}(\varepsilon,d) \ge C_p^d(1+o(1)) \quad \mbox{for $d \rightarrow \infty$}.$$ In particular, for all these $p$ the generalized $L_p$-discrepancy suffers from the curse of dimensionality.
\end{thm}

\section{The case $p=q=2$ revisited and an outlook on the general problem}\label{sec:q2}

As already mentioned, the case $p=q=2$ is well established and can be tackled with the method of decomposable kernels (see \cite{NW01} and also \cite[Chapter~12]{NW10}). This method can be mimicked with the method presented here by taking the decomposition  
\begin{eqnarray*}
h_{1,1}(x) & = & (2-a) \min(x,a),\\
h_{1,2,(0)}(x) & = & \mathbf{1}_{[0,a]}(x) x (a-x),\\
h_{1,2,(1)}(x) & = & \mathbf{1}_{[a,1]}(x)\left(x(2-x)-a(2-a) \right)
\end{eqnarray*}
for some $a \in (0,1)$ (compare this with \cite[p.~192]{NW10}). See Figure~\ref{fig_ca2}.
\begin{figure}
\begin{center}
\includegraphics[width=10cm]{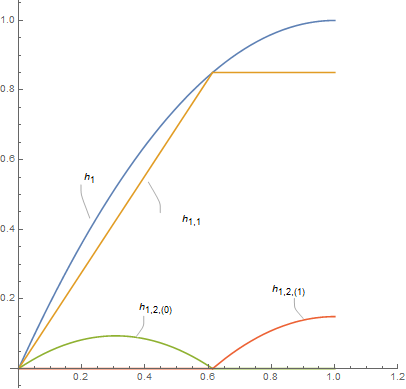}
\caption{Plot of the functions $h_1, h_{1,1}, h_{1,2,(0)}$, and $h_{1,2,(1)}$ for $q=p=2$ and $a=4^{1/3}/(1+4^{1/3})$.}
\label{fig_ca2}
\end{center}
\end{figure}
Define the fooling function $g_d$ like in \eqref{def:fool}, but with the present decomposition, it follows almost immediately from orthogonality arguments that $\|g_d\|_{d,2} \le \|h_d\|_{d,2}$. Choosing $a=4^{1/3}/(1+4^{1/3})$ (compare again with \cite[p.~192]{NW10}) gives in addition that $\int_0^1 h_{1,2,(0)}(x)\rd x = \int_0^1 h_{1,2,(1)}(x)\rd x$ such that the whole proof of Theorem~\ref{thm2} goes through with the present decomposition, but now one obtains the improved quantity $$2^{\alpha_3/(1+\alpha_3)}= 1.08332\ldots,$$ i.e., $$N_2^{{\rm int}}(\varepsilon,d) \ge (1.08332)^d (1+o(1)),$$ which is exactly the value obtained in \cite[p.~193]{NW10}. This reproves the result for $p=q=2$, but without using the method of decomposable kernels.

From this point of view, the following decomposition for general $q \in (1,\infty)$ would be somehow obvious:

\begin{eqnarray*}
h_{1,1}(x) & = & \frac{1-(1-a)^p}{a} \min(x,a),\\
h_{1,2,(0)}(x) & = & \mathbf{1}_{[0,a]}(x)\left((1-(1-x)^p)-\frac{x}{a}(1-(1-a)^p) \right),\\
h_{1,2,(1)}(x) & = & \mathbf{1}_{[a,1]}(x)\left((1-(1-x)^p)-(1-(1-a)^p) \right).
\end{eqnarray*}

Choosing $a=2^{\frac{p}{p+1}}/(1+2^{\frac{p}{p+1}})$ one obtains again that $\int_0^1h_{1,2,(0)}(x)\rd x=\int_0^1h_{1,2,(1)}(x)\rd x$. However, choosing the fooling function $g_d$ like in \eqref{def:fool} the problem is to show the estimate $\|g_d\|_{d,q} \le \|h_d\|_{d,q}$. As soon as this could be done, one can tackle the general $q$ case with the proof method from Theorem~\ref{thm2}.

\vspace{0.5cm}
\noindent{\bf Author's Address:}\\

\noindent Erich Novak, Mathematisches Institut, FSU Jena, Ernst-Abbe-Platz 2, 07740 Jena, Germany. Email: erich.novak@uni-jena.de\\

\noindent Friedrich Pillichshammer, Institut f\"{u}r Finanzmathematik und Angewandte Zahlentheorie, JKU Linz, Altenbergerstra{\ss}e 69, A-4040 Linz, Austria. Email: friedrich.pillichshammer@jku.at


\begin{thebibliography}{10}
\bibitem{A11} C. Aistleitner: Covering numbers, dyadic chaining and discrepancy. J. Complexity 27 (6): 531--540, 2011.

\bibitem{BC} J. Beck and W.W.L. Chen: Irregularities of Distribution. Cambridge University Press, Cambridge, 1987.

\bibitem{BLV} D. Bilyk, M.T. Lacey, and A. Vagharshakyan: On the small ball inequality in all dimensions. J. Funct. Anal. 254 (9): 2470--2502, 2008. 

\bibitem{C80} W.W.L. Chen: On irregularities of distribution. Mathematika 27: 153--170, 1981.

\bibitem{C83} W.W.L. Chen: On irregularities of distribution II. Quart. J. Math. Oxford 34: 257--279, 1983.

\bibitem{CS02} W.W.L. Chen and M.M. Skriganov: Explicit constructions in the classical mean squares problem in irregularities of point distribution. J. Reine Angew. Math. 545: 67--95, 2002.

\bibitem{D56} H. Davenport: Note on irregularities of distribution. Mathematika 3: 131--135, 1956.

\bibitem{D14} J. Dick: Discrepancy bounds for infinite-dimensional order two digital sequences over $\mathbb{F}_2$. J. Number Theory 136: 204--232, 2014.

\bibitem{DHPP} J. Dick, A. Hinrichs, F. Pillichshammer, and J. Prochno: Tractability properties of the discrepancy in Orlicz norms. J. Complexity 61, paper ref. 101468, 9 pp., 2020.

\bibitem{DKP} J. Dick, P. Kritzer, and F. Pillichshammer: Lattice Rules -- Numerical Integration, Approximation, and Discrepancy. Springer Series in Computational Mathematics 58, Springer, Cham, 2022.

\bibitem{DP14a} J. Dick and F. Pillichshammer: Optimal $\mathcal{L}_2$ discrepancy bounds for higher order digital sequences over the finite field $\mathbb{F}_2$. Acta Arith. 162: 65--99, 2014.

\bibitem{D21} B. Doerr: A sharp discrepancy bound for jittered sampling.  Math. Comp. 91(336): 1871--1892, 2022.

\bibitem{DT97} M. Drmota and R.F. Tichy: Sequences, Discrepancies and Applications. Lecture Notes in Mathematics 1651, Springer Verlag, Berlin, 1997.

\bibitem{GH21} M. Gnewuch and N. Hebbinghaus: Discrepancy bounds for a class of negatively dependent random points including Latin hypercube samples.  Ann. Appl. Probab. 31(4): 1944--1965, 2021. 

\bibitem{GPW} M. Gnewuch, H. Pasing, and Ch. Weiss: A generalized Faulhaber inequality, improved bracketing covers, and applications to discrepancy. Math. Comp. 90 (332): 2873--2898, 2021. 

\bibitem{hnww} S. Heinrich, E. Novak, G. Wasilkowski, and H. Wo\'{z}niakowski: The inverse of the star-discrepancy depends linearly on the dimension. Acta Arith. 96(3): 279--302, 2001.

\bibitem{Hi04} A. Hinrichs: Covering numbers, Vapnik-\v{C}ervonenkis classes and bounds for the star-discrepancy. J. Complexity 20(4): 477--483, 2004. 

\bibitem{HKNV21} 
A. Hinrichs, D. Krieg, E. Novak and J. Vyb\'\i ral: 
Lower bounds for the error of quadrature formulas for Hilbert spaces, 
J. Complexity 65, paper ref. 101544, 20 pp., 2021. 


\bibitem{HKNV22} 
A. Hinrichs, D. Krieg, E. Novak and J. Vyb\'\i ral: 
Lower bounds for integration and recovery in $L_2$, 
J. Complexity 72, paper ref. 101662, 15 pp., 2022.   


\bibitem{KV23} 
D. Krieg and J. Vyb\'\i ral: 
New lower bounds for the integration of periodic functions. 
Preprint arXiv 2302.02639. 

\bibitem{kuinie} L. Kuipers and H. Niederreiter: Uniform Distribution of Sequences. John Wiley, New York, 1974.

\bibitem{M15} L. Markhasin: $L_p$- and $S_{p,q}^rB$-discrepancy of (order $2$) digital nets. Acta Arith. 168: 139--159, 2015.

\bibitem{mat} J. Matou\v{s}ek: Geometric Discrepancy -- An Illustrated Guide, Algorithms and Combinatorics, 18,  Springer-Verlag, Berlin, 1999.

\bibitem{NoX} E. Novak:
Intractability results for positive quadrature formulas and extremal 
problems for trigonometric polynomials. 
J. Complexity 15: 299--316, 1999.

\bibitem{NSW97} E. Novak, I.H. Sloan and H. Wo\'zniakowski: 
Tractability of tensor product linear operators, 
J. Complexity 13: 387--418, 1997. 



\bibitem{NW01} E. Novak and H. Wo\'{z}niakowski: Intractability results for integration and discrepancy. J. Complexity 17(2):  388--441, 2001.

\bibitem{NW08} E. Novak and H. Wo\'{z}niakowski: Tractability of Multivariate Problems. Volume I: Linear Information. EMS Tracts in Mathematics 16, Z\"urich, 2008.

\bibitem{NW10} E. Novak and H. Wo\'{z}niakowski: Tractability of Multivariate Problems. Volume II: Standard Information for Functionals. EMS Tracts in Mathematics 12, Z\"urich, 2010.

\bibitem{PW20} H. Pasing and C. Weiss: Improving a constant in high-dimensional discrepancy estimates. Publ. Inst. Math. (Beograd) (N.S.) 107(121): 67--74, 2020.

\bibitem{P22} F. Pillichshammer: The BMO-discrepancy suffers from the curse of dimensionality. J. Complexity 76, paper ref. 101739, 7 pp., 2023.

\bibitem{Roth} K.F. Roth: On irregularities of distribution. Mathematika 1: 73--79, 1954.

\bibitem{R80} K.F. Roth: On irregularities of distribution. IV. Acta Arith. 37: 67--75, 1980.

\bibitem{S06} M.M. Skriganov: Harmonic analysis on totally disconnected groups and irregularities of point distributions. J. Reine Angew. Math. 600: 25--49, 2006.

\bibitem{SW97}  I.H. Sloan and H. Wo\'{z}niakowski: 
An intractability result for multiple integration. 
Math. Comp. 66: 1119--1124, 1997.  


\bibitem{SW1998}  I.H. Sloan and H. Wo\'{z}niakowski: When are quasi-Monte Carlo algorithms efficient for high-dimensional integrals? J. Complexity 14: 1--33, 1998.

\bibitem{sch77} W.M. Schmidt: Irregularities of distribution. X. In: Number Theory and Algebra, pp. 311--329, Academic Press, New York, 1977.

\bibitem{Vy20}  
J. Vyb\'\i ral: 
A variant of Schur's product theorem and its applications. 
Adv. Math. 368, paper ref. 107140, 9 pp., 2020. 

\bibitem{Wo99} H. Wo\'{z}niakowski: Efficiency of quasi-Monte Carlo algorithms for high dimensional integrals. In: Monte Carlo and Quasi-Monte Carlo Methods 1998 (H. Niederreiter and J. Spanier, eds.), pp.  114--136, Springer Verlag, Berlin, 2000.
\end{thebibliography}
\end{document}